\documentclass[a4paper,11pt]{amsart}
\usepackage[latin1]{inputenc}

\usepackage{amsfonts, amssymb, amsmath, amscd, latexsym, graphicx, psfrag, color, float, xspace, enumerate}
\usepackage[all]{xy}

\usepackage{todonotes}
\usepackage{hyperref}
\usepackage{mathrsfs}

\usepackage{fullpage} 

\newenvironment{enumerate1}
{\begin{enumerate}[\upshape (1)]}
{\end{enumerate}}

\newtheorem{dummy}{dummy}[section]
\newtheorem{lemma}[dummy]{Lemma}
\newtheorem{theorem}[dummy]{Theorem}
\newtheorem{innercustomthm}{Theorem}
\newenvironment{customthm}[1]
  {\renewcommand\theinnercustomthm{#1}\innercustomthm}
  {\endinnercustomthm}

\newtheorem{proposition}[dummy]{Proposition}
\theoremstyle{definition}
\newtheorem{definition}[dummy]{Definition}
\newtheorem{example}[dummy]{Example}

\newtheorem{remark}[dummy]{Remark}

\newcommand{\bA}{\mathbb{A}}
\newcommand{\bC}{\mathbb{C}}

\newcommand{\bG}{\mathbb{G}}
\newcommand{\bH}{\mathbb{H}}

\newcommand{\bN}{\mathbb{N}}

\newcommand{\bQ}{\mathbb{Q}}
\newcommand{\bR}{\mathbb{R}}
\newcommand{\bZ}{\mathbb{Z}}

\newcommand{\cA}{\mathcal{A}}

\newcommand{\cH}{\mathcal{H}}

\newcommand{\cO}{\mathcal{O}}

\newcommand{\cR}{\mathcal{R}}

\newcommand{\cX}{\mathcal{X}}
\newcommand{\cY}{\mathcal{Y}}

\newcommand{\Spec}{\mathrm{Spec}\,}

\newcommand{\Hom}{\mathrm{Hom}}

\newcommand{\Div}{\mathrm{Div}}

\renewcommand{\log}{\mathrm{log}}

\newcommand{\gp}{\mathrm{gp}}
\newcommand{\an}{\mathrm{an}}
\newcommand{\class}{\mathrm{B}}

\newcommand{\et}{\mathrm{\acute{e}t}}
\renewcommand{\top}{\mathrm{top}}
\newcommand{\mon}{\mathrm{Mon}}

\newcommand{\bRp}{\bR_{\geq 0}}
\newcommand{\bRps}{\bRp\times S^1}

\newcommand{\Logst}{\mathfrak{Logst}}
\newcommand{\Topst}{\mathfrak{Topst}}

\newcommand\radice[2][\relax]{\hspace{-1.5pt}\sqrt[\uproot{2}#1]{#2}}

\newcommand{\Sch}{\mathrm{Sch}}

\begin{document}

\title[The Kato-Nakayama space as a transcendental root stack]{The Kato-Nakayama space \\ as a transcendental root stack}

\author{Mattia Talpo}
\address{Department of Mathematics\\
University of British Columbia\\
1984 Mathematics Road\\
Vancouver, BC, V6T 1Z2\\
Canada}
\email{mtalpo@math.ubc.ca}

\author[Angelo Vistoli]{Angelo Vistoli$^\dagger$}
\address{Scuola Normale Superiore\\Piazza dei Cavalieri 7\\
56126 Pisa\\ Italy}

\email{angelo.vistoli@sns.it}

\thanks{$^\dagger$Partially supported by research funds from the Scuola Normale Superiore}

\maketitle

\begin{abstract}
We give a functorial description of the Kato-Nakayama space of a fine saturated log analytic space, that is similar in spirit to the functorial description of root stacks. As a consequence we get a global description of the comparison map constructed in \cite{knvsroot} from the Kato-Nakayama space to the (topological) infinite root stack.\end{abstract}

\setcounter{tocdepth}{1}

\tableofcontents

\section{Introduction}

Let $X$ be a fine saturated log scheme, locally of finite type over $\bC$, or a log analytic space. There have been a few constructions aimed at capturing the ``log geometry'' of $X$ in more familiar forms. Two of those are the ``Kato-Nakayama'' space $X_{\log}$ (a topological space, introduced in \cite{KN}), and the ``infinite root stack'' $\radice[\infty] X$ (a pro-algebraic stack, introduced in \cite{TV}). As mentioned in the introduction of \cite{TV}, the latter is, morally, an ``algebraic version'' of the former.

Building on this idea, in the paper \cite{knvsroot} by Carchedi, Scherotzke, Sibilla and the first author it is shown that there is a canonical morphism $\Phi_X\colon X_\log\to \radice[\infty]{X}_\top$ from the Kato-Nakayama space to the topological realization of the infinite root stack of $X$, that is moreover a ``profinite equivalence''. In that paper the morphism is constructed locally on $X$, in presence of a Kato chart for the log structure, and then globalized by gluing \cite[Section 4]{knvsroot}.

Two natural questions arise from this construction.
\begin{enumerate1}
\item Is there a global definition of the morphism $\Phi_X$? For example, one could hope to use a ``functor of points'' point of view, describing the objects of the groupoid  $\radice[\infty]{X}_\top$, and then producing an object of $\radice[\infty]{X}_\top(X_\log)$.

\item Over $X_\log$ there is a sheaf of rings $\cO_X^\log$ that makes the projection $X_\log\to X$ into a map of ringed spaces, and over $\radice[\infty]{X}_\top$ there is a natural structure sheaf  $\cO_\infty$. Does $\Phi_X$ extend to a morphism of ringed topological stacks?
\end{enumerate1}

In this paper we answer the first question, and give such a global construction of $\Phi_{X}$. Question $(2)$ is addressed in the recent preprint \cite{parabolic} by the first author, where a notion of ``coherent sheaf'' on $X_\log$ and the relationship with parabolic sheaves with real weights are discussed.

In order to answer question $(1)$ we will give a functorial description of $\radice[\infty]{X}_\top$, and produce an object of the corresponding kind on the topological space $X_\log$. We will do so by giving a ``root stack'' functorial definition of $X_{\log}$, that is closely related to the one given in \cite{illusie-kato-nakayama}. From this it will be apparent that an object parametrized by the Kato-Nakayama space induces compatible $n$-th roots for each positive integer $n$.

In some more detail: we use the point of view of \cite{borne-vistoli}, according to which a log structure on $X$ can be seen as a symmetric monoidal functor $L\colon A\to [\bC/\bC^\times]_X$ from a sheaf of monoids to a stack of line bundles with global section on open subsets of $X$. Here the monoial structure on $[\bC/\bC^\times]_X$ is given by tensor product of line bundles with a section, and the functor $L$ is compatible with this structure. Recall also that for $n\in \bN$, the root stack $\radice[n]{X}$ parametrizes liftings of the functor $L$ along the $n$-th power map $\wedge n\colon [\bC/\bC^\times]\to [\bC/\bC^\times]$, induced by $z\mapsto z^n$ on both the space $\bC$ and the group $\bC^\times$ (and corresponding to raising both the line bundle and the global section to the $n$-th power). 

The Kato-Nakayama space turns out to parametrize similar liftings, in which instead of extracting $n$-th roots for a fixed $n$ we are in some sense extracting a ``logarithm'', i.e. we are lifting the log structure along a sort of ``exponential'' $\cH\to [\bC/\bC^\times]$, where $\cH$ is the stack constructed as follows.

{Let $\overline{\bC}$ be the topological monoid $(\{-\infty\}\cup \bR)\times \bR$, where the operation is addition on both factors, and we declare that $-\infty+x=-\infty$ for every $x\in \{-\infty\}\cup \bR$. Using the exponential $\{-\infty\}\cup \bR\xrightarrow{\cong} \bR_{\geq 0}$ in the first factor, the monoid $\overline{\bC}$ can also be seen as the ``closed right half-plane'' $\bH=\bR_{\geq 0}\times \bR\subseteq \bC$, equipped with the operation $(x,y)\cdot (x',y')=(xx',y+y')$. The additive group $\bC^+$ of complex numbers acts on $\overline{\bC}$ by ``translation'', i.e. as $(a+ib)\cdot (x,y)=(a+x, b+y)$, compatibly with the monoid structure. We have an exponential map $\overline{\bC}\to \bC$ sending $(x,y)$ to $e^{x+iy}$ (which is $0$ if $x=-\infty$), that is $\exp$-equivariant, for $\exp\colon \bC^+\to \bC^\times$ the usual exponential. The stack $\cH$ alluded to above is the quotient $\cH=[\overline{\bC}/\bC^+]$, that by the preceding observations admits a map $\exp\colon \cH\to [\bC/\bC^\times]$.}

The following theorem, that gives a functorial description of the Kato-Nakayama space as a sort of ``transcendental root stack'', is our main result.

\begin{customthm}{A}[Theorem \ref{theorem:comparison}]\label{thm:A}
Let $X$ be a fine saturated log analytic space, with DF structure $L\colon A\to [\bC/\bC^\times]_X$. Then the stack on topological spaces over $X_\top$ that sends $f\colon T\to X_\top$ to the groupoid of symmetric monoidal functors $f^{-1}A \to \cH_T$ lifting the functor $f^{-1}L\colon f^{-1}A\to [\bC/\bC^\times]_T$ is represented by the Kato-Nakayama space $X_\log$.
\end{customthm}

As a consequence, we obtain, in Section \ref{subsec:morphism}, a global description of the canonical morphism  $\Phi_X\colon X_\log\to \radice[\infty]{X}_\top$ constructed in \cite[Section 4]{knvsroot}: for every $n$ there is a factorization
   $$
   \cH\to [\bC/\bC^\times]\stackrel{\wedge n}{\longrightarrow}[\bC/\bC^\times]
   $$
 of the map $\exp\colon \cH\to [\bC/\bC^\times]$, given by the morphisms $\overline{\bC}\to \bC$ sending $(x,y)$ to $e^{(x+iy)/n}$ and $\bC^+\to \bC^\times$ defined as $z\mapsto e^{z/n}$. Consequently, because of Theorem \ref{thm:A}, for every object of the groupoid $X_\log(T)$ we obtain a compatible system of objects of the groupoids $\radice[n]{X}_\top(T)$, i.e. an object of $\radice[\infty]{X}_\top$ over $T$. This describes the morphism $\Phi_X$ in functorial terms.
 
In the last paragraph,  $\radice[n]{X}_\top$ denotes the underlying ``topological stack'' of the $n$-th root stack of $X$. In order to obtain a functorial description of these stacks, we have to develop a bit of theory for a kind of ``complex-valued'' log structures on topological spaces. It turns out, in fact, that the topological stack $\radice[n]{X}_\top$ coincides with the $n$-th root stack (in this theory of ``log topological spaces'') of the log topological space $X_\top$ (Proposition \ref{prop:root.comparison}).
 
Finally, in proving that the morphism $X_\log\to \radice[\infty]{X}_\top$ that we obtain coincides with the one of \cite{knvsroot} (which we do in Proposition \ref{prop:morphism}), we also point out, in Section \ref{sec:5}, that the Kato-Nakayama construction can be applied to log algebraic (or analytic) stacks, by mimicking the construction of the analytification functor for stacks (recalled briefly in Section \ref{sec:analytic.stacks}).

\subsection*{Outline}

Section \ref{sec:log.str.analytic.top} contains the basics of log structures on analytic and topological spaces, both in the language of Kato \cite{kato} and in the alternative ``Deligne--Faltings'' language introduced in \cite{borne-vistoli}. We describe spaces and stacks of charts, and consider root stacks in this general framework, proving in particular that the formation of root stacks is compatible with the analytification and ``underlying topological space'' functors.

In Section \ref{sec:4} we describe our functorial interpretation of the Kato-Nakayama space, that is the translation in the Deligne-Faltings language of the one given in \cite[Section 1]{illusie-kato-nakayama} (recalled in this paper as Theorem \ref{thm.functorial}). We describe the natural ``charts'' for $X_\log$ that correspond to this description, and we produce a globally defined morphism from the Kato-Nakayama space to the topological infinite root stack.

To conclude, in Section \ref{sec:5} we prove that the Kato-Nakayama construction can be extended to algebraic (and analytic) stacks, and we check that the morphism to the infinite root stack produced in the previous section coincides with the one of \cite{knvsroot}.

\subsection*{Acknowledgments}

We are happy to thank Kai Behrend for a key idea, and David Carchedi,
Nicol\`{o} Sibilla and Jonathan Wise for useful conversations. {We are also grateful to the anonymous referee for useful comments and suggestions.}

\subsection*{Notations and conventions}

We assume some familiarity with log geometry. For an introduction, see for example \cite{kato} or \cite[Appendix]{knvsroot}. We are mostly interested in fine and saturated log structures.

By the results of \cite{borne-vistoli}, for schemes the ``Kato language'' is equivalent to the  ``Deligne-Faltings'' language.

All our monoids will be commutative. If $P$ is a monoid and $X$ is a monoid with some additional structure (for example a topological space), we will denote by $X(P)$ the object $\Hom_\mon(P,X)$ with its naturally induced additional structure. For example we can take $X=\bRp$ to be the topological monoid of non-negative real numbers with respect to multiplication, and then $\bRp(P)$ will denote the topological monoid $\Hom(P,\bRp)$. We will denote by $\widehat{P}$ the diagonalizable group scheme $\Spec \bC[P^\gp]$ associated with the abelian group $P^\gp$. The sheafification of the constant presheaf with sections $P$ will be denoted by $\underline{P}$.

For symmetric monoidal categories we adopt the language and conventions of \cite{borne-vistoli} (see in particular Section 2.4).

All our algebraic spaces will be locally separated. If $X$ is a scheme (or algebraic space) over $\bC$, we write $X_\et$ for the small \'etale site of $X$. {If $X$ is an analytic (resp. topological) space we will denote by $\cA_X$ the small analytic (resp. classical) site of $X$}. If $X$ is a locally separated algebraic space that is locally of finite type over $\bC$, we will denote by $X_\an$ its analytification as an analytic space, and by $X_\top$ the underlying topological space of $X_\an$. Although $X_\top$ and $X_\an$ are the \emph{same} topological space, we usually prefer to keep the two symbols distinct, so that it will be clear whether we are in the analytic or topological world. 

We will denote by $\cO_X$ the structure sheaf of either a scheme (or algebraic space) or of an analytic space.


\section{Log structures on analytic and topological spaces}\label{sec:log.str.analytic.top}

In order to give a functorial interpretation of the topological infinite root stack $\radice[\infty]{X}_\top=\varprojlim_n \radice[n]{X}_\top$ (whose definition is recalled later) of a fine saturated log  analytic space $X$, we need to introduce a notion of log structures on a topological space. The analytic space $X$ itself could be of the form $Y_\an$ for a fine saturated log algebraic space $Y$ locally of finite type over $\bC$, so we also take the intermediate step of discussing log structures on analytic spaces in the language of \cite{borne-vistoli}.

The definitions and facts of this section can be formulated in the language of topoi with a sheaf of monoids (in the style of \cite{molcho}, that discusses log structures in the sense of Kato on certain categories of ``spaces''). A detailed treatment employing this language will appear elsewhere.

The proofs in this section will be somewhat terse. The interested reader can look at the more detailed treatment of \cite{borne-vistoli}, in the algebraic case.

In this section $X$ will be either a complex analytic space, or a topological space. We will denote by $\cA_X$ the classical site of $X$ (i.e. the site whose objects are open subsets of $X$, maps are inclusions and coverings are families of jointly surjective maps), and by $\cO_X$ the sheaf of rings of complex analytic functions in the analytic case, and of continuous complex-valued function in the topological case. We will use the term ``line bundle'' to indicate holomorphic line bundles and continuous complex line bundles, respectively.

\begin{remark}
Log structures in the analytic context have already been considered in the literature, see for example \cite{illusie-kato-nakayama}, and our notion coincides with the usual one. Log structures in a topological setting were considered, with a different spirit, in \cite{rognes}. We do not know what kind of relations there are between Rognes's definition and ours, if any.
\end{remark}

\subsection{Log and DF log structures}

The definitions that follow are the immediate generalization to our context of the ones of \cite{kato} and \cite{borne-vistoli}.

\begin{definition}\label{def.log.an.2}
A \emph{log structure} on $X$ is a sheaf of monoids $M$ on $\cA_X$ together with a map of sheaves of monoids $\alpha\colon M\to \cO_X$ that induces an isomorphism $\alpha|_{\alpha^{-1}(\cO_X)^\times}\colon \alpha^{-1}\cO_X^\times\stackrel{\cong}{\longrightarrow} \cO_X^\times$.
\end{definition}

It turns out that, as in the algebraic context, in presence of the mild assumption of quasi-integrality, this definition of a log structure is equivalent to the following.

Note that the quotient stack $\Div_X=[\cO_X/\cO_X^\times]$ on the site $\cA_X$ has a symmetric monoidal structure, induced by multiplication on $\cO_X$. It is moreover easy to check that it parametrizes pairs $(L,s)$ of a holomorphic (or continuous complex) line bundle with a global section, in analogy with the algebraic case, and the monoidal operation is identified with tensor product.

Later on, when we want to stress that we are considering things in the topological setting, we will denote the sheaf of continuous complex-valued functions on the topological space $T$ by $\bC_T$, and the stack $\Div_T$ by $[\bC/\bC^\times]_T$.

\begin{definition}\label{def.log.an.1}
A \emph{DF structure} on $X$ is a sheaf of sharp monoids $A$ on $\cA_X$ together with a symmetric monoidal functor $A\to \Div_X$ with trivial kernel.
\end{definition}

In this definition and from now on ``DF'' stands for ``Deligne--Faltings'', and ``trivial kernel'' means that if a section $a$ maps to an object that is isomorphic to $(\cO_X,1)$ (the unit object of $\Div_X$), then $a=0$. This is a particular instance of a ``Deligne-Faltings object'' as defined in \cite[Section 2]{borne-vistoli}.

One can define a category of log structures and a category of DF structures. A morphism will in both cases consist of a homomorphism of sheaves of monoids that is compatible with the structure map to $\cO_X$ ($\Div_X$ respectively, in the 2-categorical sense). Moreover, log structures and DF structures can be pulled back along morphisms of analytic or topological spaces. We refer the reader to \cite[Section 3]{borne-vistoli} for a detailed treatment, that also adapts to the present case.

Recall that a log structure is \emph{quasi-integral} if the action of $\cO_X^\times$ on $M$ is free.

\begin{proposition}\label{prop:katovsDF}
Let $X$ be an analytic (or topological) space. Then there is an equivalence of categories between quasi-integral log structures and DF structures on $X$.
\end{proposition}

\begin{proof}
The proof is a straightforward adaptation of the one of \cite[Theorem 3.6]{borne-vistoli}.
\end{proof}

As in the algebraic case, the proof shows that in comparing these two structures, the sheaf $A$ is identified with the \emph{characteristic sheaf} $\overline{M}=M/\cO_X^\times$.

\begin{definition}
A \emph{log analytic space} (resp. \emph{log topological space}) is an analytic space (resp. topological space) $X$ with a quasi-integral log structure $\alpha\colon M\to \cO_X$ (equivalently, with a DF structure $L\colon A \to \Div_X$).
\end{definition}

For the rest of the paper all Kato log structures will be quasi-integral,  we will drop the ``Kato'' and ``DF'', and just talk about log structures, and we will switch freely between the two notions and notations.

One defines morphisms of log analytic (or topological) spaces as in the algebraic case, by requiring a morphism $f\colon X\to Y$ together with a  map $f^{-1}M_Y\to M_X$, which is compatible with the morphisms to the structure sheaves. A morphism $f\colon X\to Y$ of log analytic (or topological) spaces is \emph{strict} if the map $f^{-1}M_Y\to M_X$ is an isomorphism.

\subsection{Charts}

Let us discuss local models for log structures in our context. The arguments of Sections 3.3 and 3.4 in \cite{borne-vistoli}  can be adapted without difficulties, but we will refrain from giving a fully detailed treatment.

Let $P$ be a finitely generated monoid. The analytic space $(\Spec \bC[P])_\an$ admits a ``tautological'' log structure, induced via sheafification by the map of monoids $P\to \bC[P]$. This coincides with the ``divisorial'' log structure induced by the open embedding $\widehat{P}_\an\subseteq (\Spec \bC[P])_\an$, i.e. the log structure obtained by considering the subsheaf $M$ of $\cO_X$ of functions that are invertible on $\widehat{P}_\an$ (we are using the notation $\widehat{P}$ of \cite{borne-vistoli} for the Cartier dual of $P^\gp$). The log structure is equivariant for the action of $\widehat{P}_\an$, and hence induces a log structure on the quotient stack $[(\Spec \bC[P])_\an/\widehat{P}_\an]$ (see Section \ref{sec:analytic.stacks} below for a brief reminder about analytic and topological stacks). In the topological setting, we can consider the underlying topological spaces $(\Spec \bC[P])_\top$ and $\widehat{P}_\top$, and the analogous quotient stack $[(\Spec \bC[P])_\top/\widehat{P}_\top]$. These objects will also be equipped with tautological log structures.

\begin{remark}
Strictly speaking, we have not defined log structures on analytic or topological stacks, but we trust that the reader interested in the subtlety will be able to fill the gap. For example, they can be seen as systems of compatible log structures on analytic (or topological) spaces mapping to the given stack, {as in \cite[Definition 2.10]{logmckay}}.\end{remark}

In order to uniformize the notation, in this section we will generally denote by $\bA(P)$ the analytic space $(\Spec \bC[P])_\an$ (reps. the topological space $(\Spec \bC[P])_\top$), and by $\cA(P)$ the quotient stack $[(\Spec \bC[P])_\an/\widehat{P}_\an]$ (resp. $[(\Spec \bC[P])_\top/\widehat{P}_\top]$).

These log structures on the stack $\cA(P)$ have a more natural interpretation in terms of DF structures. The following is the analogue of \cite[Proposition 3.25]{borne-vistoli}.

\begin{lemma}\label{lemma:DFchart}
Let $P$ be a fine sharp monoid, and $X$ an analytic (or topological) space. Then there is an equivalence between the category of maps $X\to \cA(P)$ and the category of symmetric monoidal functors $P\to \Div_X(X)$.
\end{lemma}

Here $\Div_X(X)$ denotes the symmetric monoidal category of sections of the stack $\Div_X$ on the whole space $X$.

\begin{proof}
The case $P=\bN$ is clear from the fact that $[\cO_X/\cO_X^\times](X)$ is the category of line bundles with a section on $X$, and the same objects are parametrized by morphisms to $[\bC/\bC^\times]$. More generally, since $[\bC^k/(\bC^\times)^k]\cong [\bC/\bC^\times]\times \cdots \times [\bC/\bC^\times]$ where the product has $k$ factors, the conclusion follows also for $P=\bN^k$.

Next we show how to associate a symmetric monoidal functor $P\to \Div_X(X)$ to a map $X\to \cA(P)=[(\Spec \bC[P])_\an/\widehat{P}_\an]$ (we will use the notation for the analytic case - the topological case is analogous). Given $p\in P$, consider the submonoid $j\colon \langle p\rangle \subseteq P$ generated by $p$. Since $P$ is fine and sharp, $\langle p\rangle \cong \bN$. Consider the composite
$$
X\to [(\Spec \bC[P])_\an/\widehat{P}_\an]\to [\bC/\bC^\times]
$$
where the second map is induced by the inclusion $j$. This corresponds to an object $(L_p,s_p)$ of $\Div_X(X)$.

Consider now two elements $p, q\in P$. The object $(L_{p+q},s_{p+q})$ is determined by the morphism $$X\to  [(\Spec \bC[P])_\an/\widehat{P}_\an]\to [(\Spec \bC[\langle p+q \rangle])_\an/\widehat{\langle p+q \rangle}_\an]\cong [\bC/\bC^\times].$$
Now note that the map $[(\Spec \bC[P])_\an/\widehat{P}_\an]\to [(\Spec \bC[\langle p+q \rangle])_\an/\widehat{\langle p+q \rangle}_\an]$ is obtained from the ones corresponding to $p$ and $q$, by mapping to
\begin{align*}
[(\Spec \bC[\langle p,q \rangle])_\an/\widehat{\langle p,q \rangle}_\an] & \subseteq[(\Spec \bC[\langle p \rangle])_\an/\widehat{\langle p \rangle}_\an]\times [(\Spec \bC[\langle q \rangle])_\an/\widehat{\langle q \rangle}_\an]\\
&\cong  [\bC/\bC^\times]\times [\bC/\bC^\times]
\end{align*}
(where $\langle p,q \rangle$ denotes the submonoid of $P$ generated by $p$ and $q$), and then further to the quotient stack $[(\Spec \bC[\langle p+q \rangle])_\an/\widehat{\langle p+q \rangle}_\an]\cong  [\bC/\bC^\times]$, via $\otimes\colon   [\bC/\bC^\times]\times [\bC/\bC^\times]\to [\bC/\bC^\times]$.

Consequently, we obtain an isomorphism $(L_p,s_p)\otimes (L_q,s_q)\cong (L_{p+q},s_{p+q})$. This gives a symmetric monoidal structure to the assignment $p\mapsto (L_p,s_p)$, and we obtain a symmetric monoidal functor $P\to \Div_X(X)$.

To go in the opposite direction, let us take a presentation $f\colon \bN^r\to P$ with a finite number of relations $s_i=t_i$, where $s_i, t_i\in \bN^r$ for $i=1,\hdots,N$ (a finitely generated monoid is also finitely presented - this is R\'edei's theorem \cite[Theorem 72]{redei}). In other words $P$ is the coequalizer $\bN^N\rightrightarrows\bN^r\to P$ in the category of commutative monoids.

The given functor $P\to \Div_X(X)$ induces a functor $\bN^r \to \Div_X(X)$ such that the two composites $\bN^N\to \Div_X(X)$ are isomorphic. The case $P=\bN^k$ gives us a morphism $X\to [\bC^r/(\bC^\times)^r]$, with an isomorphism between the two composites $X\to  [\bC^N/(\bC^\times)^N]$.

Now we point out that the diagram
$$
[(\Spec\bC[P])_\an/\widehat{P}_\an]\to [\bC^r/(\bC^\times)^r]\rightrightarrows [\bC^N/(\bC^\times)^N]
$$
is an equalizer in analytic stacks. This follows from the algebraic analogue of what we are proving (which is Proposition 3.25 of \cite{borne-vistoli}) and the fact that the analytification functor (on algebraic stacks) preserves finite limits. This gives a morphism $X\to [(\Spec\bC[P])_\an/\widehat{P}_\an]$. One easily checks that the resulting functor is a quasi-inverse to the previous construction.
\end{proof}


The previous lemma gives the quotient stack $\cA(P)$ a universal DF structure (in both the analytic and topological cases).

\begin{definition}
A \emph{Kato chart} for $X$ is a strict morphism $X\to \bA(P)$. A \emph{DF chart} for $X$ is a strict morphism $X\to \cA(P)$.
\end{definition}


A Kato chart gives a DF chart by composing with the projection $\bA(P) \to \cA(P)$ (which is strict).

As in the algebraic case, one can check that a symmetric monoidal functor $P\to \Div_X(X)$, induces by sheafification a DF structure $A_P\to \Div_X$. The sheaf $A_P$ is obtained from the constant sheaf $\underline{P}$ on $X$ by killing the local sections that become invertible in $\Div_X$, so the map $\underline{P}\to A_P$ is a cokernel in the category of sheaves of monoids on $\cA_X$. This is, in fact, the definition of a chart in \cite[Section 3.3]{borne-vistoli}.

Analogously, a Kato chart corresponds to a homomorphism of monoids $P\to \cO_X(X)$ that induces the given log structure $M\to \cO_X$ by sheafifying to $\underline{\alpha}\colon \underline{P}\to \cO_X$, and taking the associated log structure $\underline{P}\oplus_{\underline{\alpha}^{-1}\cO_X^\times} \cO_X^\times\to \cO_X$.

\begin{definition}
A log analytic (or topological) space is \emph{coherent} if it locally admits Kato charts for finitely generated monoids.
\end{definition}

We will assume that all our log structures are coherent, and add adjectives such as ``fine'' and ``saturated'' with the usual meaning, i.e. that one can find charts with monoids $P$ that have the corresponding property. This will be equivalent to ask for the stalks of the characteristic monoid $\overline{M}$ to have the corresponding property (see \cite[Section 3.3]{borne-vistoli} for details).

\begin{remark}
One can check that for fine saturated log structures (in a quite general setting), locally admitting Kato charts is equivalent to locally admitting DF charts.
\end{remark}

\subsection{Analytic and topological stacks}\label{sec:analytic.stacks}

Let us briefly pause to recall the notions of analytic and topological stacks, and the extension of the analytification functor on schemes over $\bC$. We refer the reader to \cite{No1} for details, especially about the latter case.

As algebraic stacks over schemes on some base $S$ are defined as categories fibered in groupoids over $(\Sch/S)$ that satisfy a gluing condition and are presented by a groupoid $R\rightrightarrows U$ with ``nice'' structure maps (typically \'etale or smooth), analytic and topological stacks are defined in the same way by switching schemes with the appropriate kind of object.

For analytic stacks we consider the site of analytic spaces with the classical topology, and consider stacks that are presented by groupoids $R\rightrightarrows U$ where the structure maps are holomorphic submersions. We will use the term ``Deligne--Mumford'' to indicate stacks that can be presented with a groupoid where the structure maps are \'etale.

For the topological case, a topological stack will be a stack on the site of topological spaces with the classical topology, and admitting a presentation by a groupoid $R\rightrightarrows U$ with structure maps that are ``locally cartesian maps with Euclidean fibers'' - the analogue in this context of smooth maps (see \cite{No1}). We will say that a topological stack is ``Deligne--Mumford'' if it can be presented by a groupoid with \'etale structure maps (i.e. local homeomorphisms).

There is an analytification functor that produces an analytic stack from an algebraic stack (locally of finite type over $\bC$), and an ``underlying topological stack'' functor that produces a topological stack from an analytic stack. They both extend the natural analytification functor on schemes of finite type over $\bC$ and ``underlying topological space'' functor on analytic spaces, respectively.

\begin{remark}\label{rmk:analytification.functor}
Let us briefly sketch the construction of the analytification functor (the other case is analogous), and refer the reader to \cite[Section 20]{No1} for more details. We will apply the same process to the ``Kato-Nakayama functor'' in Section \ref{sec:5} in order to extend it to log algebraic stacks, and give a slightly more detailed proof (see Theorem \ref{thm:functor}).

Given an algebraic stack $\cX$ locally of finite type over $\bC$, we want to produce an analytic stack $(\cX)_\an$. Let us choose a presenting groupoid $R\rightrightarrows U$ for $\cX$, and consider the induced groupoid $R_\an\rightrightarrows U_\an$. This is a groupoid in analytic spaces, whose structure maps are holomorphic submersions.
Hence the quotient $[R_\an/U_\an]$ is an analytic stack, that we take to be the analytification $(\cX)_\an$. One can check that the construction does not depend on the presenting groupoid (up to unique isomorphism), and that this extends to a functor from algebraic stacks to analytic stacks.

A more conceptual proof can be given along the lines of \cite[Theorem 3.1]{knvsroot}, by constructing $(\cX)_\an$ via the left Kan extension of $(-)_\an$ along the Yoneda embedding. This gives for $(\cX)_\an$ the ``explicit'' formula
$$
(\cX)_\an=\varinjlim_{\Spec R\to \cX} (\Spec R)_\an
$$
where the colimit is a lax colimit in the 2-category of analytic stacks.
\end{remark}

\subsection{Root stacks}\label{sec:root.stacks}

Let us briefly discuss root stacks (\cite[Section 4]{borne-vistoli}) in the two settings analyzed in the previous sections. 

Given a sheaf of fine saturated monoids $A$ on $X$ and $n\in \bN$, consider the inclusion $i_n\colon A\to \frac{1}{n}A$. This can be also be seen as the map $A\to A$ that multiplies sections by $n$.

\begin{definition}
Let $X$ be a log analytic (or topological) space, and $n\in \bN$ a positive natural number.

The $n$-th root stack of $X$ is defined by assigning to an analytic (or topological) space $Y$ the groupoid $\radice[n]{X}(Y)$ of triples $(\phi, N, a)$, where $\phi\colon Y\to X$ is a morphism, $N\colon  \phi^{-1}\frac{1}{n}A\to \Div_Y$ is a symmetric monoidal functor with trivial kernel and $a$ is a natural equivalence from $\phi^{-1}L$ to the composite $N\circ i_n$. The arrows are the obvious ones.
\end{definition}

One easily checks that the formation of root stacks is compatible with strict base change.
Note also that If $n\mid m$ there is a natural projection $\radice[m]{X}\to \radice[n]{X}$, induced by the factorization $\frac{1}{n}A\subseteq \frac{1}{m}A$ of $i_m\colon A\to \frac{1}{m}A$.

\begin{definition}
The infinite root stack $\radice[\infty]{X}$ of $X$ is the inverse limit $\varprojlim_n \radice[n]{X}$.
\end{definition}

This object can be seen either as a pro-object, or as a stack over the category of analytic (or topological) spaces, although in the analytic case it is better to see it as a pro-object (see Remark \ref{rmk:too-small}).

As a stack, $\radice[\infty]{X}$ can also be seen as functorially parametrizing symmetric monoidal functors $A_\bQ\to \Div_X$ that extend the given $L\colon A \to \Div_X$, where $A_\bQ$ is the union $\bigcup_{n\in \bN} \frac{1}{n}A$ of all the Kummer extensions of $A$. See \cite[Section 3]{TV} for details.

\begin{proposition}\label{prop:rootDM}
Let $X$ be a fine saturated log analytic (or topological) space. Then for every $n$ the $n$-th root stack $\radice[n]{X}$ is an analytic (or topological) Deligne--Mumford stack.
\end{proposition}

\begin{proof}
The same proof given in \cite[Section 4]{borne-vistoli} applies. We briefly sketch it below.

We can assume that there is a global Kato chart $X\to \bA(P)$. Because this map is strict, the diagram
$$
\xymatrix{
\radice[n]{X}\ar[r]\ar[d] & \radice[n]{\bA(P)}\ar[d]\\
X\ar[r] & \bA(P)
}
$$
is cartesian, so it suffices to show that $\radice[n]{\bA(P)}$ is an analytic DM stack.

From the functorial definition and Lemma \ref{lemma:DFchart} it is clear that the root stack of the stack $\cA(P)$ is the stack $\cA(\frac{1}{n}P)$, with the natural map $\cA(\frac{1}{n}P)\to \cA(P)$. Note that since $\frac{1}{n}P\cong P$ as monoids, we actually have $\cA(\frac{1}{n}P)\cong \cA(P)$ as stacks with a log structure, but the map mentioned above is not the identity.

Moreover, the diagram
$$
\xymatrix{
\radice[n]{\bA(P)}\ar[r]\ar[d]& \cA(\frac{1}{n}P)\ar[d]\\
\bA(P) \ar[r] &\cA(P)
}
$$
is cartesian, and shows that, as in the algebraic case, we have an isomorphism
$$
\textstyle \radice[n]{\bA(P)}=\left[\bA(\frac{1}{n}P)/\mu_n(P)\right],
$$
where $\mu_n(P)$ is the Cartier dual (i.e. the group of characters) of the cokernel of $P^\gp\to \frac{1}{n}P^\gp$.
\end{proof}

Note that this also gives a quotient stack presentation
$$
\radice[n]{X}\cong [X_n/(\mu_n(P))_\an]
$$
in presence of a Kato chart $X\to \bA(P)$, where $X_n=X \times_{\bA(P)}\bA(\frac{1}{n}P)$. The infinite root stack, as in the algebraic case, is only pro-Deligne--Mumford.

\begin{remark}\label{rmk:too-small}
There is a difference in the analytic case, regarding the infinite root stack $\radice[\infty]{X}$, that is worth a few words. In the analytic case, this stack has very few objects. This is due to the fact that analytic spaces are by definition locally of finite type, and thus they cannot have roots of every order of a non-zero non-constant holomorphic function.

Let us consider for example $X=\bC$, with coordinate $z$, as a log analytic space with the log structure given by the origin, and its infinite root stack $\radice[\infty]{X}$. If $f\colon Y\to \bC$ is a non-constant map from an analytic space $Y$ that hits the origin, then the pullback to $Y$ of the function $z$ is a non-constant analytic function, and there cannot exist a sequence of analytic functions $z_n$ on $Y$  such that $z_n^n=f^*z$ for all $n$: the local ring $\cO_{Y,y}$ in a point $y\in Y$ that maps to the origin is local Noetherian, and if there existed roots as above, then $f^*z\in \mathfrak{m}_y^n$ for every $n$, hence we would have $f^*z=0$, as $\bigcap_{n}  \mathfrak{m}_y^n =\{0\}$.

In this case $\radice[\infty]{X}$ is isomorphic to the disjoint union $\bC^\times \bigsqcup \radice[\infty]{0}$, where the origin $0\in \bC$ is given the induced log structure. Because of this, in the analytic setting it is best to see the infinite root stack as a pro-object instead than an actual stack.
\end{remark}

\subsection{Comparison of root stacks}\label{sec:3}

Let us compare the different notions of log structures and root stacks that we just defined, using the analytification and ``underlying topological space'' functors.

If $X$ is a fine saturated log algebraic space locally of finite type over $\bC$, then by applying the analytification functor we obtain an analytic space $X_\an$, and an induced analytic log structure. This is obtained by pulling back via the natural morphism of ringed topoi $\phi\colon X_\an\to X_\et$, where $X_\et$ is the small \'etale topos of $X$.

Concretely, given $\alpha\colon M\to \cO_X$ on $X_\et$, we obtain a sheaf $\phi^{-1}M$ on the analytic site of $X_\an$, and an induced symmetric monoidal functor $\phi^{-1}\alpha\colon \phi^{-1}M\to \cO_{X_\an}$. This is not a log structure because $\cO_{X_\an}^\times$ is bigger than $\phi^{-1}\cO_X^\times$, but we can take the associated log structure
$$
\alpha_\an\colon M_\an= \phi^{-1}M\oplus_{(\phi^{-1}\alpha)^{-1}\cO_{X_\an}^\times}\cO_{X_\an}^\times\to \cO_{X_\an}
$$
(note that the sheaf $\overline{M}$ does not change, i.e. $\overline{M_\an}=\phi^{-1}\overline{M}$). Hence we can analytify a fine saturated log scheme locally of finite type to obtain a fine saturated log analytic space.

In the same way, starting from a fine saturated log analytic space and applying the ``underlying topological space'' functor, we obtain a fine saturated log topological space. It is clear that both of these operations preserve the existence of local charts, and the sheaf $\overline{M}$. Consequently, properties of the log struxcture such as being finitely generated, integral, saturated or coherent are also preserved.

We prove now that the three versions of the root stack construction (algebraic, analytic and topological) are compatible with the analytification and ``underlying topological space'' functors.

\begin{proposition}\label{prop:root.comparison}
Let $X$ be a fine saturated  log algebraic space locally of finite type over $\bC$ (resp. analytic space), and $n\in\bN$ be a positive integer.

Then the analytic (resp. topological) stack $\radice[n]{X}_\an$ (resp. $\radice[n]{X}_\top$)  associated with the $n$-th root stack of $X$ is canonically isomorphic to the $n$-th root stack $\radice[n]{X_\an}$ (resp.  $\radice[n]{X_\top}$) of the associated log analytic space (resp. log topological space) of $X$.
\end{proposition}

In short, $\radice[n]{X}_\an\cong \radice[n]{X_\an}$ and  $\radice[n]{X}_\top\cong \radice[n]{X_\top}$. This will be used to describe functorially the topological infinite root stack $\radice[\infty]{X}_\top=\varprojlim_n \radice[n]{X}_\top$ of a log analytic space (or log algebraic space) $X$.

\begin{proof}
The proof will be entirely analogous in the two cases, so we will carry it out only in the analytic case.

The general construction of $(\cX)_\an$, if $\cX$ is any stack over schemes, is as a left Kan extension of $(-)_\an$ along the Yoneda embedding, {as recalled at the end of Remark \ref{rmk:analytification.functor}}. In other words we have the formula
$$
(\cX)_\an=\varinjlim_{\Spec R\to \cX} (\Spec R)_\an
$$
where the colimit is a lax colimit in the 2-category of analytic stacks.

Now assume we are given a map $f\colon \Spec R\to \radice[n]{X}$, and let us explain how to produce a map $g\colon (\Spec R)_\an\to \radice[n]{X_\an}$. By the functorial description of {the stack $\radice[n]{X}$}, the map $f$ corresponds to a morphism $\phi\colon \Spec R\to X$ and a lifting of $\phi^{-1}L\colon \phi^{-1}A\to\Div_R$ to $N\colon \phi^{-1}\frac{1}{n}A\to \Div_R$, and analogously for $\radice[n]{X_\an}$ and the map $g$ on the analytic side. Hence, to obtain a map $g$ as above we need to produce a lifting $N_\an\colon \phi^{-1}\frac{1}{n}A_\an\to \Div_{(\Spec R)_\an}$ of the functor $\phi^{-1}L_\an\colon \phi^{-1}A_\an\to\Div_{(\Spec R)_\an}$.

Let $a$ be a section of $\phi^{-1}\frac{1}{n}A_\an$ over some analytic open $U\subseteq (\Spec R)_\an$. Then we can find an \'etale $V\to \Spec R$ with a section $\sigma\colon U\to V_\an$ that is a homeomorphism onto the image, and a section $b$ of $\phi^{-1}\frac{1}{n}A(V)$ that corresponds to $a$. The section $b$ gives $N(b)=(L_b,s_b)$, a line bundle over $V$ with a global section $s$. By analytifying, we get a complex line bundle $(L_b)_\an$ with a global holomorphic section $(s_b)_\an$. By restricting to $U$, this defines the image of $a$ in $\Div_{(\Spec R)_\an}(U)$. This process extends in the obvious way to a symmetric monoidal functor of monoidal stacks over the analytic site $\cA_{(\Spec R)_\an}$ that lifts $\phi^{-1}L_\an$, i.e. a morphism $(\Spec R)_\an\to \radice[n]{X_\an}$.

{From this procedure we obtain a morphism of analytic stacks
$$
\varinjlim_{\Spec R\to \radice[n]{X}} (\Spec R)_\an=\radice[n]{X}_\an \to \radice[n]{X_\an}
$$
as follows. For every $\bC$-algebra $R$ of finite type, every morphism $\Spec R\to \radice[n]{X}$ induces a morphism $(\Spec R)_\an\to \radice[n]{X_\an}$ as explained above. Moreover this assignment is compatible with commuting triangles
$$
\xymatrix{
\Spec R\ar[r]\ar[d] & \radice[n]{X}\\
\Spec R'\ar[ru]& }
$$
and therefore, by the universal property of the colimit, we obtain the desired morphism $\radice[n]{X}_\an\to \radice[n]{X_\an}$.}

To check that this map is an isomorphism, we can do so locally on $X$, where there is a Kato chart $X\to \Spec \bC[P]$ for a fine torsion-free monoid $P$. In that case we have a quotient stack description of $\radice[n]{X}$ as
$$
\radice[n]{X}=[X_n/\mu_n(P)]
$$
where $X_n=X\times_{\Spec \bC[P]}\Spec \bC[\frac{1}{n}P]$ and $\mu_n(P)$ is the Cartier dual of the cokernel of $P^\gp\to \frac{1}{n}P^\gp$, that acts on $X_n$ by acting on the second factor. From the construction of $(-)_\an$ via presenting groupoids of \cite[Theorem 20.1]{No1}, recalled in {Remark {\ref{rmk:analytification.functor}}}, it follows that $\radice[n]{X}_\an=[(X_n)_\an/\mu_n(P)_\an]$.

From the analytic stack description of $\radice[n]{X_\an}$ as a quotient in the presence of a global chart given in Section {\ref{sec:root.stacks}}, we see that it coincides with the one just described. The map $\radice[n]{X}_\an\to \radice[n]{X_\an}$ in this local case is an isomorphism, and this concludes the proof. 
\end{proof}

\section{The Kato-Nakayama space as a ``root stack''}\label{sec:4}

Let $X$ be a fine saturated log analytic space. In this section we give a functorial description of the Kato-Nakayama space $X_\log$ (see \cite{KN} or the Appendix of \cite{knvsroot}) of $X$ in the language of DF structures, and that bears a close similarity to the description of root stacks. It presents the Kato-Nakayama space as a sort of ``transcendental'' root stack.

As a byproduct of this alternative description we obtain a global construction of the canonical morphism $\Phi_X\colon X_\log\to (\radice[\infty]{X})_\top$ of \cite{knvsroot} (Section \ref{subsec:morphism} and Proposition \ref{prop:morphism} below).

Let us start by briefly recalling how $X_\log$ is constructed \cite[Section 1]{KN}. Let us denote by $p^\dagger$ the log analytic space whose underlying space is $\Spec\bC$, and the log structure is defined by the monoid $\bR_{\geq 0}\times S^1$ and the map $\alpha\colon \bR_{\geq 0}\times S^1\to \bC=\cO_{\Spec\bC}$ given by $(r,a)\mapsto r\cdot a$. One defines $X_\log$ as the set of morphisms of log analytic spaces $\Hom(p^\dagger,X)$. Equivalently, elements of $X_\log$ are pairs $(x,\phi)$ consisting of a point $x\in X$ and a homomorphism of groups $\phi\colon M_x^\gp\to S^1$ such that $\phi(f)=\frac{f(x)}{|f(x)|}$ for every $f\in \cO_{X,x}^\times\subseteq M_x^\gp $.

If $X=\bC(P)=(\Spec\bC[P])_\an$, then $X_\log$ can be naturally identified with $\Hom(P,\bR_{\geq 0}\times S^1)$. More generally, if $X$ has a Kato chart $X\to \bC(P)$, then $X_\log$ can be identified with a closed subset of the space $X\times \Hom(P^\gp,S^1)$ (where $\Hom(P^\gp,S^1)$ has its natural topology), and we can equip it with the induced topology. This turns out to be independent of the particular Kato chart that we choose, so we get a topology on the space $X_\log$ for a general $X$. 

The resulting map $\tau\colon X_\log\to X$ that sends $(x,\phi)$ to $x$ is continuous and proper. The fiber $\tau^{-1}(x)$ over a point $x\in X$ can be identified with the space $\Hom(\overline{M}_x^\gp,S^1)$, which is non-canonically isomorphic to a real torus $(S^1)^r$, where $r$ is the rank of the (finitely generated) free abelian group $\overline{M}_x^\gp$. If the log structure of $X$ is determined by a normal crossings divisor $D\subseteq X$, then the space $X_\log$ coincides with the ``real oriented blowup'' of $X$ along $D$.

The space $X_\log$ should be thought of as an ``underlying topological space'' of the log analytic space $X$, where the log structure is replaced by the non-trivial topology of the fibers of the map $\tau\colon X_\log\to X$. For example, in \cite[Theorem 0.2]{KN} it is proven that log \'etale and log de Rham cohomology on $X$ can be identified with ``Betti'' (or singular) cohomology on the space $X_\log$.

\subsection{The case of a single divisor}\label{sec:one.divisor}

Let us first look at a motivating example. 

Assume that $X$ is a smooth analytic space with a log structure given by a single smooth divisor. In this case there is a global chart $X\to [\bC/\bC^\times]$ for the log structure, corresponding to the map $\bN\to [\bC/\bC^\times](X)$ that sends $1$ to $(\cO_X(D),1_D)$. Here we are considering $\bC$ and $\bC^\times$ as analytic spaces.

The various root stacks of $X$ can be obtained as fibered products in the following manner (see Section {\ref{sec:root.stacks}}): if $\wedge n\colon  [\bC/\bC^\times]\to [\bC/\bC^\times]$ is the map induced by ``raising to the $n$-th power'' on both the space and the group, we have a cartesian diagram
$$
\xymatrix{
\radice[n]{X}\ar[r]\ar[d] & [\bC/\bC^\times]\ar[d]^{\wedge n}\\
X\ar[r] &  [\bC/\bC^\times].
}
$$
The basic insight is that the Kato-Nakayama space can be obtained in a similar way as well. The idea for what follows is due to Kai Behrend.

Let us consider the ``extended complex plane'' $$\overline{\bC}=(\{-\infty\}\times \bR)\cup \bC=(\{-\infty\}\cup \bR)\times \bR$$
with its operation given by addition (where $-\infty+x=-\infty$ for every $x\in \{-\infty\}\cup \bR$), that makes it a commutative topological monoid. There is an action of the group $\bC^+$ of complex numbers with addition (we use this notation to distinguish it from the analytic space $\bC$) on $\overline{\bC}$ given by translation, i.e. $(a+ib)\cdot (x,y)=(a+ x, b+y)$, and we will consider the quotient stack $[\overline{\bC}/\bC^+]$ as a topological stack. 

The action of $\bC^+$ on $\overline{\bC}$ has two orbits: points $(x,y)$ with $x\in \bR$ have trivial stabilizer and the action is transitive among them, so they give a single open point of $[\overline{\bC}/\bC^+]$. The other orbit is the line $\{-\infty\}\times \bR$, with stabilizer $\bR^+\subseteq \bC^+$. So we can loosely write $[\overline{\bC}/\bC^+]=*\cup \class\bR^+$.

We have a morphism of stacks $\exp\colon [\overline{\bC}/\bC^+]\to [\bC/\bC^\times]$ given by the exponential $\mathrm{exp}\colon \bC^+\to \bC^\times$ at the level of groups, and by the $\mathrm{exp}$-equivariant map $\overline{\bC}\to \bC$ that sends $(x,y)$ to $e^{x+iy}$ (with the convention that $e^{-\infty +iy}=0$), at the level of spaces. This coincides with the universal cover of $\bC^\times$ if we restrict it to the complement of the line $\{-\infty\}\times \bR$, which in turn gets contracted to the origin in $\bC$. Now note that $[\bC/\bC^\times]$ also has two points, namely $[\bC/\bC^\times]=*\cup \class \bC^\times$, and the morphism $[\overline{\bC}/\bC^+]\to [\bC/\bC^\times]$ ``maps'' $*$ to $*$ and $\class \bR^+\to \class \bC^\times$, via $\mathrm{exp}\colon \bR^+\to \bC^\times$. This last homomorphism is injective with cokernel isomorphic to ${S}^1$.

Because of this description, the morphism $[\overline{\bC}/\bC^+]\to [\bC/\bC^\times]$ is an isomorphism over the open point and an $S^1$-bundle over the closed point. Since the map $X\to [\bC/\bC^\times]$ sends $X\setminus D$ to the open point and $D$ to the closed point, it is apparent that by pulling back we will find precisely the Kato-Nakayama space (i.e. the real oriented blow up, in this case), so that there should be (see Section \ref{subsec:kncharts} below for the proof) a cartesian diagram
$$
\xymatrix{
X_\log \ar[r]\ar[d]_\tau &  [\overline{\bC}/\bC^+]\ar[d]^{\exp}\\
X\ar[r] & [\bC/\bC^\times].
}
$$
Moreover note that $[\overline{\bC}/\bC^+]\to [\bC/\bC^\times]$ factors as $[\overline{\bC}/\bC^+]\to  [\bC/\bC^\times] \stackrel{\wedge n}{\longrightarrow} [\bC/\bC^\times]$ for every $n$ by sending $(x,y)\in \overline{\bC}$ to $e^{(x+iy)/n}\in\bC$ and using $\mathrm{exp}\left(\frac{\cdot}{n}\right)\colon \bC^+\to \bC^\times$ on the groups.

This will give a morphism $X_\log\to \radice[n]{X}_\top$ for every $n$ (here we are using Proposition {\ref{prop:root.comparison}}), that all together will give a morphism $X_\log\to \varprojlim_n (\radice[n]{X})_\top=\radice[\infty]{X}_\top$ of topological stacks, in this special case.

\subsection{The general case}\label{subsection:in.general}

Let us use the language of DF structures to generalize the above construction. 

The log structure of $X$ is given by a morphism  $A\to \Div_X$ of symmetric monoidal stacks on the analytic site $\cA_X$. 
The $n$-th root stack $\radice[n]{X}$ parametrizes liftings of the log structure to the sheaf of formal fractions $\frac{1}{n}A$, i.e. diagrams
$$
\xymatrix{
A\ar[r]\ar[d] &\Div_X\\
\frac{1}{n}A \ar@{-->}[ru] &
}
$$
(over some analytic space over $X$) or, alternatively, liftings
$$
\xymatrix{
A\ar[r]\ar@{-->}[rd] &\Div_X\\
 &\Div_X\ar[u]_{\wedge n}
}
$$
where $\wedge n\colon \Div_X \to \Div_X$ sends $(L,s)$ into $(L^{\otimes n}, s^{\otimes n})$.

There is a description of the Kato-Nakayama space in this spirit, that uses the symmetric monoidal stack $[\overline{\bC}/\bC^+]$ introduced in the previous section (which turns out to ``dominate'' every such root morphism $\wedge n$, as we already explained above and will discuss in more detail in Section \ref{subsec:morphism}).

\begin{definition}
Let us consider the stack $X_{\overline{\bC}}$ over the category of \emph{topological spaces} over $X_\top$ that sends a space $\phi\colon T\to X_\top$ to the groupoid of liftings
$$
\xymatrix{
\phi^{-1}A\ar[r]\ar@{-->}[rd] & [\bC/\bC^\times]_T\\
 & [\overline{\bC}/\bC^+]_T \ar[u]_{\exp}
}
$$
where $\phi^{-1}A\to [\overline{\bC}/\bC^+]_T$ is a symmetric monoidal functor. The arrows between the objects are given by the obvious natural transformations.
\end{definition}
Here the map $\phi^{-1}A\to[\bC/\bC^\times]_T$ is the pullback to $T$ of the topological DF structure on $X_\top$ induced by the given analytic DF structure on $X$.

The stack $X_{\overline{\bC}}$ parametrizes liftings of the $\bC^\times$-torsors $(\phi^{-1}L)(a)$ to $\bC^+$-torsors along $\exp\colon \bC^+\to \bC^\times$, equipped with a $\bC^+$-equivariant map to $\overline{\bC}$ that covers the given $\bC^\times$-equivariant map $(\phi^{-1}L)(a)\to \bC$.

\begin{theorem}\label{theorem:comparison}
Let $X$ be a fine saturated log analytic space. The stack $X_{\overline{\bC}}$ is represented by the Kato-Nakayama space $X_\log$, i.e. there is a canonical isomorphism of topological stacks $X_{\overline{\bC}}\cong X_\log$ over $X_\top$.
\end{theorem}

The starting point of the proof is the following functorial characterization of $X_\log$.

\begin{theorem}[{\cite[(1.2)]{illusie-kato-nakayama}}]\label{thm.functorial}
Consider the functor $F_\log$ that sends a topological space $\phi\colon T\to X_\top$ to the set of morphisms of sheaves of abelian groups $c\colon \phi^{-1}M^\gp\to S^1_T$ such that $c(\phi^{-1}f)=f/|f|$ for $f\in \cO_X^\times$, and that acts in the obvious way on the arrows.

Then $F_\log$ is represented by the Kato-Nakayama space $X_\log$.
\end{theorem}

\begin{proof}[Proof of Theorem \ref{theorem:comparison}]
Let us describe concretely how the analytic log structure $\alpha\colon M\to \cO_X$ on $X$ induces a topological log structure on $X_\top$. We take the composite $\beta\colon M\to \bC_{X_\top}$ of $\alpha$ and the natural map $\cO_X\to \bC_{X_\top}$ (here we are denoting by $\bC_T$ the sheaf of continuous complex-valued functions on the topological space $T$), and then form the amalgamated sum ${M}_\top=M\oplus_{\beta^{-1}(\bC^\times_{X_\top})}\bC^\times_{X_\top}$. The induced map $\alpha_\top \colon M_\top \to \bC_{X_\top}$ gives the topological log structure. Note that we have an isomorphism between the characteristic shaves $\overline{M}\cong \overline{M_\top}$, induced by $M\to M_\top$.

We will rephrase the functorial interpretation of Theorem \ref{thm.functorial} in the language of DF structures. First note that since $S^1$ is a group, we have $\Hom(\phi^{-1}M^\gp,S^1_T)=\Hom(\phi^{-1}M,S^1_T)$ and this is compatible with the condition on sections of $\cO_X^\times$.

We claim that the set of homomorphisms $$\left\{c\in \Hom(\phi^{-1}M,S^1_T)\;\; \Big| \;\; c(\phi^{-1}f)=\frac{f}{|f|} \mbox{ for } f\in \cO_X^\times\right\}$$ is the same as
the set of morphisms of sheaves of monoids $d\colon \phi^{-1}M\to (\bR_{\geq 0}\times S^1)_T$ such that 
$d(\phi^{-1}f)=(|f|,f/|f|)$ for $f \in \cO_X^\times$ and the composite $\phi^{-1}M\to (\bRps)_T\to (\bRp)_T$ is {the homomorphism sending a section $m\in \phi^{-1}M$ to the continuous function $|{\alpha_\top}(m)|$ with values in $\bR_{\geq 0}$} (where $|\cdot|$ denotes the usual euclidean absolute value on $\bC$).

Given such a $d$, we can compose with the second projection $(\bR_{\geq 0}\times S^1)_T\to S^1_T$ and obtain a $c\in  \Hom(\phi^{-1}M,S^1_T)$ satisfying the condition above. In the other direction, given $c:\phi^{-1}M\to S^1_T$, one can define the corresponding $d$ via $d(m)=\left(|{\alpha_\top}(m)|,c(m)\right)$.

Now we claim that morphisms $d\colon\phi^{-1}M\to (\bR_{\geq 0}\times S^1)_T$ as above correspond to symmetric monoidal functors $$\overline{d}\colon \phi^{-1}\overline{M} \to [\bR_{\geq 0}\times S^1/\bC^\times]_T$$
that lift the DF structure $L_\top\colon \phi^{-1}\overline{M}\to [\bC/\bC^\times]_T$ associated with $\alpha_\top$. Here the action of $\bC^\times\cong \bR_{> 0}\times S^1$ on $\bRps$ is given by multiplication on the two factors and $[\bR_{\geq 0}\times S^1/\bC^\times]_T\to  [\bC/\bC^\times]_T$ is induced by the $\bC^\times$-equivariant function $\bR_{\geq 0}\times S^1\to \bC$ sending $(r,a)$ to $r\cdot a\in \bC$. 

First observe that, by construction of the sheaf $M_\top$ and the log structure $\alpha_\top$, there is a bijection between maps $d\colon \phi^{-1}M\to (\bRps)_T$ such that $d(\phi^{-1}f)=(|f|, f/|f|)$ for every $f\in \cO_X^\times$ and maps $\widetilde{d}\colon \phi^{-1}M_\top\to (\bRps)_T$ such that $\widetilde{d}(f)=(|f|, f/|f|)$ for every $f\in \bC_T^\times$.

Now note that the group $\bC_T^\times$ acts on both $\phi^{-1}M_\top$ and $(\bRps)_T$, and moreover the action on $\phi^{-1}M_\top$ is free, with quotient $\phi^{-1}\overline{M}$. By taking the stacky quotient of $\widetilde{d}$ by this action we get a symmetric monoidal functor
$$
\overline{d}\colon \phi^{-1}\overline{M} \to [\bR_{\geq 0}\times S^1/\bC^\times]_T.
$$
Observe also that the composite $\phi^{-1}\overline{M} \stackrel{\overline{d}}{\longrightarrow} [\bR_{\geq 0}\times S^1/\bC^\times]_T\to [\bC/\bC^\times]_T$ is naturally identified with $\phi^{-1}L_\top$, where $L_\top$ is the DF structure associated with $\alpha_\top$.

The inverse construction is obtained by taking the base change of such a $\overline{d}$ along the projection $(\bRps)_T\to [\bR_{\geq 0}\times S^1/\bC^\times]_T$, which is a $\bC^\times_T$-torsor.

Finally we note that there is an isomorphism of symmetric monoidal stacks $$[\overline{\bC}/\bC^+]\cong [\bRps/\bC^\times],$$
where the action on the left is the same as in Section \ref{sec:one.divisor}. The subgroup $j\colon \bZ\subseteq \bC^+$ given by $k\mapsto 2k\pi i$ acts without stabilizers on $\overline{\bC}$, and the quotient is $\overline{\bC}/\bZ = \bRps$ (the map $\overline{\bC}\to \bRps$ is $(x,y)\mapsto (e^x,e^{iy})$). Moreover the cokernel of $j$ is $\bC^\times$ (and the map is given by the exponential), and therefore
$$
[\overline{\bC}/\bC^+]\cong [(\overline{\bC}/\bZ)/(\bC^+/\bZ)] \cong [\bRps/\bC^\times]
$$
as symmetric monoidal stacks.

This also gives an isomorphism of symmetric monoidal stacks $[\bR_{\geq 0}\times S^1/\bC^\times]_T\cong [\overline{\bC}/\bC^+]_T$ over the site $\cA_T$, which is compatible with the natural maps to $[\bC/\bC^\times]_T$. This shows that the functorial description of Theorem \ref{thm.functorial} coincides with the one of the stack $X_{\overline{\bC}}$ that we introduced above, and concludes the proof.
\end{proof}

\begin{remark}
With the same reasoning as in the proof, we also have
   \[
   [\bR_{\geq 0}\times S^1/\bC^\times]\cong [\bR_{\geq 0}/\bR_{>0}]
   \]
(by writing $\bC^\times=\bR_{>0}\times S^1$ and cancelling the $S^1$ factor).
\end{remark}

\subsection{Charts}\label{subsec:kncharts}

In the spirit of the functorial interpretation of Theorem \ref{theorem:comparison}, we can obtain ``charts'' for the Kato-Nakayama space $X_\log$ out of charts for the log structure of $X$.

Specifically, when $X$ has a  DF chart $X\to [(\Spec \bC[P])_\an/\widehat{P}_\an]$ with $P$ fine and torsion-free, the cartesian diagram described in Section \ref{sec:one.divisor} can be replaced by the more general
\begin{equation}\label{kn.charts}
\xymatrix{
X_\log\ar[r]\ar[d] & [\overline{\bC}(P)/\bC^+(P)]\ar[d]\\
X_\top\ar[r] & [\bC(P)/\bC^\times(P)]
}
\end{equation}
where ${\overline{\bC}}(P)=\Hom(P,\overline{\bC})$ and $\bC^+(P)=\Hom(P,\bC^+)$ have their natural topologies and monoid or group structures. The vertical map is given by composition with the exponential maps $\overline{\bC}\to \bC$ and $\bC^+\to \bC^\times$ that were discussed in Section \ref{sec:one.divisor}, and for $P=\bN$ the diagram reduces to the one showing up at the end of the discussion.

\begin{remark}
For every finitely generated monoid $P$ there is an isomorphism
   \[
   [\overline{\bC}(P)/\bC^+(P)]\cong [(\bRps)(P)/\bC^\times(P)]
   \]
induced by the projection $\overline{\bC}(P)\to (\bRps)(P)$ and the exponential $\bC^+(P)\to \bC^\times(P)$, as in the proof of Theorem \ref{theorem:comparison}. In an analogous way we also have an isomorphism
   \[
   [(\bRps)(P)/\bC^\times(P)]\cong [\bRp(P)/\bR_{>0}(P)]\,.
   \]

We can use any one of these models to describe charts for $X_\log$, and we will switch back and forth without further mention.
\end{remark}

\begin{proposition}\label{prop:kncharts}
Let $X$ be a fine saturated log analytic space equipped with a DF chart $X\to [(\Spec \bC[P])_\an/\widehat{P}_\an]$, with $P$ fine and torsion-free. Then there is a natural diagram (\ref{kn.charts}) as above, and it is cartesian.
\end{proposition}

\begin{proof}
The point is to show that $[\overline{\bC}(P)/\bC^+(P)] \to [\bC(P)/\bC^\times(P)]$ is the stack of charts for the objects parametrized by the Kato-Nakayama space, as in Theorem \ref{theorem:comparison}.

{Given the isomorphisms
   \[
   [\overline{\bC}/\bC^+]\cong [\bRps/\bC^\times]
   \]
and
   \[
   [\overline{\bC}(P)/\bC^+(P)]\cong [(\bRps)(P)/\bC^\times(P)]\,,
   \]
this claim follows from the following two facts. {
\begin{itemize}
\item For a fine torsion-free monoid $P$ and a topological space $T$, symmetric monoidal functors $P\to [\bRps/\bC^\times](T)$ correspond to morphisms $T\to [(\bRps)(P)/\bC^\times(P)]$ (the analogue of Lemma \ref{lemma:DFchart} - see also \cite[Proposition 3.10]{TVlogstr}).
\item A symmetric monoidal functor $P\to [\bRps/\bC^\times](T)$ can be sheafified to a morphism of symmetric monoidal stacks $A\to [\bRps/\bC^\times]_T$ with trivial kernel (see Propositions 2.4 and 2.10 of \cite{borne-vistoli}), compatibly with the sheafification of the induced functor $P\to [\bRps/\bC^\times](T)\to [\bC/\bC^\times](T)$.
\end{itemize}
We leave the details to the reader.}}
\end{proof}

These kinds of charts are related to the local models for $X_\log$ given by the topological  space $\Hom(P,\bR_{\geq 0}\times S^1)$ (see \cite[Section 1]{KN}) in the same way as DF charts are related to Kato charts for log spaces.

In fact for every fine monoid $P$ the natural diagram
$$
\xymatrix{
(\bRps)(P)\ar[r]\ar[d] & [(\bRps)(P)/\bC^\times(P)]\ar[d]\\
\bC(P)\ar[r] & [\bC(P)/\bC^\times(P)]
}
$$
is cartesian.

Note however that by using the presentation $[\overline{\bC}(P)/\bC^+(P)]$ to compute the fibered product above, we would ``spontaneously'' end up with the diagram
$$
\xymatrix{
[\overline{\bC}(P)/\bZ(P)]\ar[r]\ar[d] & [\overline{\bC}(P)/\bC^+(P)]\ar[d]\\
\bC(P) \ar[r] & [\bC(P)/\bC^\times(P)]
}
$$
{where the group $\bZ(P)=\Hom(P,\bZ)$ is the kernel of the surjective map $\exp(P)\colon \bC^+(P)\to \bC^\times(P)$. Note that to be precise this should be denoted by $2\pi i\bZ(P)$, and thought of as $\Hom(P,2\pi i \bZ)$, but we prefer to keep the notation lighter, and have the coefficient $2\pi i$ in the inclusion $\bZ\to \bC^+$, as $k\mapsto 2\pi i k$.}

Note that of course there is an isomorphism $(\bRps)(P)\cong [\overline{\bC}(P)/\bZ(P)]$, but this {quotient stack presentation} gives some insight into the fact that the space $\overline{\bC}(P)$, that is used by Ogus in \cite{ogus} in the form of $\bH(P)=\Hom(P,\bH)=\Hom(P,\bR_{\geq 0})\times \Hom(P,\bR)$ (see in particular Section 3.1), is like an ``atlas'' for the Kato-Nakayama space in this language. This is further explored in \cite{parabolic}, in relation to a correspondence between certain sheaves of modules on $X_\log$ and parabolic sheaves with real weights.

We can see an analogy with root stacks by looking at the presentations $\radice[n]{X}\cong [X_n/\mu_n(P)]$ given by a Kato chart $X\to \Spec \bC[P]$. Here the atlas is $X_n=X\times_{\Spec\bC[P]}\Spec\bC[\frac{1}{n}P]$, and the group $\mu_n(P)$, which is the kernel of $\bC^\times(P)\to \bC^\times(P)$ induced by $z\mapsto z^n$, is the analogue of the group $\bZ(P)$ (kernel of the exponential) above.

\begin{remark}[Differentiable structure]\label{subsection:differentiable.structure}

The space $X_\log$ has more structure than just that of a topological space, as may be apparent by staring at the charts we just described. The space $\overline{\bC}$ has a smooth (even real analytic) structure (with a boundary), that is respected by the action of $\bC^+$.

In fact, as proven in \cite[Section 6.8]{molcho}, the Kato-Nakayama space is naturally a \emph{differentiable space} \cite{gillam}, and on top of that it carries a sort of ``log structure'' of its own. Precisely, it is has a \emph{positive log differentiable structure} \cite[Section 6.1]{molcho}, meaning a log structure on the space $X_\log$ equipped with the sheaf of monoids $\cR_{X_\log}^{\geq 0}$, where for a differentiable space $Y$ the sheaf $\cR_{Y}^{\geq 0}$ on $Y$ is the sheaf of functions of differentiable spaces to $\bR_{\geq 0}$.

In analogy with the other cases {(i.e. of log structures on algebraic, analytic and topological spaces)}, the stack of DF charts for this category of log structures would be the quotient $[\bRp(P)/\bR_{>0}(P)]$ {(see \cite[Proposition 3.10]{TVlogstr} for a precise statement, in a more general setting)}, and in fact the charts for $X_\log$ described in this section are of this form. The resulting structure on $X_\log$ coincides with the one of  \cite[Section 6.8]{molcho}, since the maps $X_\log\to [\overline{\bC}(P)/\bC^+(P)]$ factor through the ``Kato chart'' given by $X_\log\to (\bRps)(P)\to \bRp(P)$, in presence of a Kato chart $X\to (\Spec\bC[P])_\an$ for $X$.

\end{remark}

\subsection{The map to the infinite root stack}\label{subsec:morphism}

Let us show how the functorial interpretation of Theorem \ref{theorem:comparison} gives a globally defined morphism of topological stacks $X_\log\to \radice[n]{X}_\top$ for every $n$, and these assemble into a morphism of pro-topological stacks $X_\log\to \radice[\infty]{X}_\top$. We will check later (see Proposition \ref{prop:morphism}) that this morphism coincides with the one constructed in \cite[Proposition 4.1]{knvsroot}.

The point here is that the description of $X_\log$ as a root stack allows us to canonically extract $n$-th roots, as follows: for a topological space $\phi\colon T\to X_\top$ let us define $$\Phi_n(T) \colon X_\log(T)\to \radice[n]{X}_\top(T)$$ by sending a morphism of symmetric monoidal categories $\phi^{-1}A\to [\overline{\bC}/\bC^+]_T$ to the composite with the map $f_n \colon [\overline{\bC}/\bC^+]_T\to [\bC/\bC^\times]_T$,  induced by $\overline{\bC}\to \bC$ that sends $(x,y)$ to  $e^{(x+iy)/n}$ and $\bC^+\to \bC^\times$ that sends $z$ to $e^{z/n}$.

This gives an object of $\radice[n]{X}_\top$. In fact we have a commutative diagram
$$
\xymatrix{
[\overline{\bC}/\bC^+]_T \ar[r]^{f_n}\ar[rd]_{\exp} & [\bC/\bC^\times]_T \ar[d]^{\wedge n}\\ 
& [\bC/\bC^\times]_T   
}
$$
that shows that $\phi^{-1}A\to [\overline{\bC}/\bC^+]_T\xrightarrow{f_n} [\bC/\bC^\times]_T$ lifts the functor $\phi^{-1}L\colon \phi^{-1}A\to [\bC/\bC^\times]_T$ along the $n$-th power map $\wedge n\colon  [\bC/\bC^\times]_T\to  [\bC/\bC^\times]_T$. Here we are using the description of the functor of points of $\radice[n]{X}_\top$ given by Proposition {\ref{prop:root.comparison}}.

The resulting morphisms are compatible with respect to the projections $\radice[m]{X}_\top\to \radice[n]{X}_\top$ where $n\mid m$, and they give a morphism of pro-objects $X_\log\to \radice[\infty]{X}_\top$.

\begin{remark}
As in the previous discussions, we can exchange $[\overline{\bC}/\bC^+]$ with $[\bRps/\bC^\times]$. Note however that the above maps cannot be defined as equivariant maps $\bRps\to \bC$ (there is no section of the maps $z^n\colon S^1\to S^1$ or $\bC^\times\to \bC^\times$). 
\end{remark}

\begin{remark}[Real roots]\label{rmk:real.roots}
The point of the above construction is to make use of the morphism $\phi_{\frac{1}{n}}\colon [\overline{\bC}/\bC^+]\to [\overline{\bC}/\bC^+]$ induced by the maps $\overline{\bC}\to \overline{\bC}$ acting as $(x,y)\mapsto (x/n,y/n)$ and $\bC^+\to \bC^+$ given by $z\mapsto {z/n}$. This corresponds to extracting $n$-th roots, in that the diagram
$$
\xymatrix{
[\overline{\bC}/\bC^+]\ar[d]_\exp & [\overline{\bC}/\bC^+]\ar[d]^\exp \ar[l]_{\phi_{\frac{1}{n}}}\\
[\bC/\bC^\times]\ar[r]^{\wedge n} & [\bC/\bC^\times]
}
$$
commutes.

More generally one can consider $\phi_r\colon [\overline{\bC}/\bC^+]\to [\overline{\bC}/\bC^+]$ for any $r\in \bR_{> 0}$, given in the same way by  the maps $\overline{\bC}\to \overline{\bC}$, defined as $(x,y)\mapsto (rx,ry)$, and $\bC^+\to \bC^+$ given by $z\mapsto {rz}$. Note that $\phi_r$ is an isomorphism for every $r$, with inverse $\phi_{\frac{1}{r}}$.

Using these morphisms one can show that a lifting parametrized by the Kato-Nakayama space
$$
\xymatrix{
\phi^{-1}A\ar[r]^<<<<<L \ar[dr]_{L_{\overline{\bC}}} & [\bC/\bC^\times]_T\\
 & [\overline{\bC}/\bC^+]_T\ar[u]
}
$$
will induce a 2-commutative diagram
$$
\xymatrix{
\phi^{-1}A\ar[d]\ar[r]^<<<<<L & [\bC/\bC^\times]_T\\
\phi^{-1}A_{\bR_{\geq 0}} \ar[ur]_{L_\bR} & }
$$
{(where $A_{\bR_{\geq 0}}$ is the subsheaf of monoids of $A^\gp\otimes_\bZ \bR$ generated by sections of the form $a\otimes r$ with $a\in A$ and $r\in \bR_{\geq 0}$)} that can be seen as a ``real root'' of the log structure, by setting
$$
L_\bR\left(\sum_i r_i\cdot a_i\right)=\exp(\phi_{r_1}(L_{\overline{\bC}}(a_1)))\otimes\cdots\otimes \exp(\phi_{r_k}(L_{\overline{\bC}}(a_k)).
$$
where $r_i \in \bRp$ and $a_i$ are sections of $\phi^{-1}A$.

It is not clear whether these two kinds of lifting can be identified completely, especially without imposing any ``continuity'' conditions on the second type of diagrams. 
\end{remark}

\section{The Kato-Nakayama ``space'' of a log algebraic stack}\label{sec:5}

In this final section we observe that the Kato-Nakayama construction applies also to log algebraic stacks that are locally of finite type over the complex numbers (and to log analytic stacks) and produces a topological stack, and we relate this to the ``charts'' for the Kato-Nakayama space described in Section \ref{subsec:kncharts}. We also check that the morphism $X_\log\to \radice[\infty]{X}_\top$ that was described in Section \ref{subsec:morphism} coincides with the one of \cite[Proposition 4.1]{knvsroot}.

Denote by $\Logst$ the 2-category of locally of finite type fine log algebraic stacks over $\bC$, and by $\Topst$ the 2-category of topological stacks.

\begin{theorem}\label{thm:functor}
There is a morphism of 2-categories $(-)_\log\colon \Logst\to \Topst$ that preserves colimits, and extends the usual Kato-Nakayama functor on log algebraic spaces. This functor is uniquely determined (in the 2-categorical sense) by these properties.
\end{theorem}

\begin{remark}
The preceding theorem is valid also if we replace the 2-category $\Logst$ by the 2-category of fine log analytic stacks. Let us note that the analytification $LOG_\an$ of Olsson's stack $LOG$ (see \cite{Ols}) is the stack that parametrizes fine log structures on analytic spaces, that we denote temporarily  by $LOG_\bC$. 

In fact, the discussion in \cite[Section 5]{Ols} describes the stack $LOG$ as the colimit in the category of stacks of the diagram, indexed by finitely generated integral monoids, of the toric stacks $[\Spec\bC[P]/\widehat{P}]$, with the natural maps between them. Since the local toric models are ``the same'', that discussion applies also to the analytic stack $LOG_\bC$ that parametrizes fine log structures on analytic spaces, which is then the colimit, indexed by the same category, of the stacks $[(\Spec\bC[P])_\an/\widehat{P}_\an]$.

Finally the analytification functor preserves colimits, and there is a natural isomorphism $[\Spec\bC[P]/\widehat{P}]_\an\cong [(\Spec\bC[P])_\an/\widehat{P}_\an]$ for any fine monoid $P$, so we obtain an induced canonical isomorphism $LOG_\an\cong LOG_{\bC}$.

The functor $(-)_\log$ can be applied to the analytic stack $LOG_\an$ to obtain a ``universal'' Kato-Nakayama space, in the sense that for every fine log analytic space (or stack) $X$ there is a cartesian square of topological stacks
$$
\xymatrix{
X_\log\ar[r]\ar[d] & LOG_\log\ar[d]\\
X_\top \ar[r] & LOG_\top.
}
$$
\end{remark}

\begin{remark}
As it happens for the Kato-Nakayama space (see Remark \ref{subsection:differentiable.structure}), for every log algebraic (or analytic) stack $\cX$ the topological stack $\cX_\log$ should have a structure of a ``(positive log) differentiable stack'' over the real numbers (the terminology is a bit awkward, since ``differentiable stack'' already has a meaning in the smooth differentiable world).
\end{remark}

\begin{proof}[Proof of Theorem \ref{thm:functor}]

We mimic the proof in \cite[Section 20]{No1}.

Assume that $R\rightrightarrows U$ is a presentation of a log algebraic stack $\cX$. Then both $U$ and $R$ have induced log structures, and the structure maps of the groupoid presentation are strict. Since the Kato-Nakayama functor preserves finite limits, we see that the resulting $R_\log\rightrightarrows U_\log$ is a groupoid in topological spaces.

Since the structure maps of $R\rightrightarrows U$ are smooth and strict, the structure maps of $R_\log\rightrightarrows U_\log$ are ``locally cartesian maps with euclidean fibers'' in Noohi's terminology (see  \cite{No1}). 
We define $\cX_\log$ to be the quotient stack $[U_\log/R_\log]$.
We sketch an argument to justify that this is independent of the groupoid presentation and extends to a functor, leaving most of the 2-categorical details to the reader.

Given another presenting groupoid $R'\rightrightarrows U'$ of $\cX$, we can find a third one $R^{''}\rightrightarrows U^{''}$ that has a map to both of these, inducing isomorphisms between the associated stacks. We will check that the induced morphisms $[U^{''}_\log/R^{''}_\log]\to [U^{'}_\log/R^{'}_\log]$ and $[U^{''}_\log/R^{''}_\log]\to [U_\log/R_\log]$ are isomorphisms. In this way we get an isomorphism $[U^{'}_\log/R^{'}_\log]\to [U_\log/R_\log]$ that depends on the choice of the third groupoid, but is unique up to a unique isomorphism. This defines the functor on objects.

Let us check that a map of groupoids $(R\rightrightarrows U)\to (R'\rightrightarrows U')$ in algebraic spaces that induces an isomorphism of quotient stacks gives an isomorphism $[U_\log/R_\log]\cong [U'_\log/R'_\log]$. We use the following fact (see \cite[Tag 046R]{stacks-project}): a morphism of groupoids as above in an arbitrary site induces an isomorphism between the quotient stacks if and only if
\begin{itemize}
\item [(i)] the composite $t\circ \pi_1\colon R'\times_{U'}U\to U'$ locally admits sections, and
\item [(ii)] the natural map $R\to (U\times U)\times_{U'\times U'} R'$ is an isomorphism.
\end{itemize}
In our situation, since $(R\rightrightarrows U)\to (R'\rightrightarrows U')$ induces an isomorphism on the quotient stacks, we infer that (i) and (ii) hold. Using the fact that all maps are strict and by applying the functor $(-)_\log$ to all diagrams, we can conclude that (i) and (ii) also hold for the map of groupoids $(R_\log \rightrightarrows U_\log)\to (R'_\log \rightrightarrows U'_\log)$ in topological spaces. Thus, the induced $[U_\log/R_\log]\to [U'_\log/R'_\log]$ is an isomorphism.

On 1-arrows, given a morphism $f\colon \cX\to \cY$ we can find presenting groupoids of $\cX$ and $\cY$ and a map between those, that induces $f$. We use the above construction to obtain a morphism $f_\log\colon \cX_\log\to \cY_\log$, and this again turns out to be unique up to a unique isomorphism.

The effect on natural transformation is uniquely determined by the above. 
\end{proof}

Note that by construction for any log algebraic (or analytic) stack $\cX$ there is a projection $\tau\colon \cX_\log\to \cX_\top$.

\begin{remark}
A more general statement could be proved along the lines of Theorem 3.1 (and the discussion that follows) of \cite{knvsroot}, using more sophisticated machinery.
\end{remark}

\begin{example}
The Kato-Nakayama space of $[\bA^1/\bG_m]$ is the topological stack $[\bR_{\geq 0}\times S^1/\bC^\times]$, for the usual action. 
More generally, for a fine saturated monoid $P$ the Kato-Nakayama space of the quotient $[(\Spec \bC[P])_\an/\widehat{P}_\an]$ is the stack $[(\bR_{\geq 0}\times S^1)(P)/\bC^\times(P)]$ that appears in the discussion of Section \ref{subsec:kncharts}.
\end{example}

Let us show that the morphism constructed in Section \ref{subsec:morphism} coincides with the morphism $\Phi_X$ of \cite[Section 4]{knvsroot}.

\begin{proposition}\label{prop:morphism}
The morphism $X_\log\to \radice[\infty]{X}_\top$ of Section \ref{subsec:morphism} coincides with the morphism $\Phi_X$ constructed in \cite[Proposition 4.1]{knvsroot}.
\end{proposition}

\begin{proof}
First,  in light of the construction of Theorem \ref{thm:functor} we can reinterpret Proposition 4.4 of \cite{knvsroot} as follows: if $X$ is a fine saturated log algebraic space locally of finite type over $\bC$ (or a fine saturated log analytic space), then for every positive integer $n$ the canonical morphism $(\radice[n]{X})_\log\to X_\log$ is an isomorphism.

Indeed, it is proven there that, locally where $X$ has a Kato chart, and thus the root stack has a presentation $\radice[n]{X}=[X_n/\mu_n(P)]$, the map $(X_n)_\log\to X_\log$ is a $\mu_n(P)_\an$-torsor. On the other hand the stack $(\radice[n]{X})_\log$ is identified with the quotient stack $[(X_n)_\log/\mu_n(P)_\an]$ (note $\mu_n(P)_\log=\mu_n(P)_\an$ since the log structure is trivial), and hence turns out to be isomorphic to $X_\log$ via the natural map.

Note that since for every $n$ there is a projection $(\radice[n]{X})_\log\to\radice[n]{X}_\top$, this produces a canonical morphism $X_\log\cong (\radice[n]{X})_\log\to\radice[n]{X}_\top$, that manifestly coincides with the $\Phi_n$ constructed in \cite[Section 4]{knvsroot}.

Now let us check that it also agrees with the natural transformation described in Section \ref{subsec:morphism}.

The point is the following commutative diagram:
$$
\xymatrix{
&(\radice[n]{X})_\log \ar[ld]_{\cong} 
\ar[dd]|!{[d]}\hole  \ar[rr] & &[\overline{\bC}(\frac{1}{n}P)/\bC^+(\frac{1}{n}P)]\ar[dd] \ar[dl]^{\phi_n(P)}_{\cong}\\ 
X_\log\ar[dd]\ar[rr] \ar@{-->}[rd] && [\overline{\bC}(P)/\bC^+(P)]\ar[dd] \ar@/_2.0pc/@{..>}[ru]_{\phi_{\frac{1}{n}}(P)} & \\
& \radice[n]{X}_\top \ar[rr]|!{[r]}\hole  \ar[ld]& &  [\bC(\frac{1}{n}P)/\bC^\times(\frac{1}{n}P)] \ar[ld]\\
X_\top \ar[rr]
& & [\bC(P)/\bC^\times(P)] & 
}
$$
where {the bottom, front and back side of the cube (and hence also the top) are cartesian, and} $\phi_n(P), \phi_{\frac{1}{n}}(P)$ are defined in the same way as the analogous maps of Remark \ref{rmk:real.roots}.

The morphism $X_\log\to \radice[n]{X}_\top$ described in Section \ref{subsec:morphism} is determined (using the functorial interpretation of $X_\log$) by the composite $$\textstyle X_\log\to [\overline{\bC}(P)/\bC^+(P)]  \xrightarrow{\phi_{\frac{1}{n}}(P)} [\overline{\bC}(\frac{1}{n}P)/\bC^+(\frac{1}{n}P)]\to [\bC(\frac{1}{n}P)/\bC^\times(\frac{1}{n}P)]$$ as the induced map $$X_\log\to X_\top\times_{[\bC(P)/\bC^\times(P)]} [\bC(\tfrac{1}{n}P)/\bC^\times(\tfrac{1}{n}P)]= \radice[n]{X}_\top.$$
On the other hand, the morphism of \cite[Proposition 4.1]{knvsroot}, as per the discussion at the beginning of the proof, is determined by the composite $$X_\log\to (\radice[n]{X})_\log\to \radice[n]{X}_\top$$
where the first arrow is the inverse of the isomorphism in the diagram above.

This implies that the two maps $X_\log\to \radice[n]{X}_\top$ that we just described coincide (in the 2-categorical sense). In fact, it is equivalent to check that the two maps $(\radice[n]{X})_\log\to \radice[n]{X}_\top$ given by the vertical arrow and by the composite $(\radice[n]{X})_\log\to X_\log\to \radice[n]{X}_\top$ coincide. The vertical arrow is determined as the map $(\radice[n]{X})_\log\to  X_\top\times_{[\bC(P)/\bC^\times(P)]} [\bC(\tfrac{1}{n}P)/\bC^\times(\tfrac{1}{n}P)]= \radice[n]{X}_\top$ induced by the two maps
\begin{equation}\label{eq:100}
(\radice[n]{X})_\log\to [\overline{\bC}(\tfrac{1}{n}P)/\bC^+(\tfrac{1}{n}P)]\to  [\bC(\tfrac{1}{n}P)/\bC^\times(\tfrac{1}{n}P)]
\end{equation}
and 
$$
(\radice[n]{X})_\log\to X_\log\to X_\top
$$
and the second one is determined likewise, by the two maps
\begin{equation}\label{eq:200}
(\radice[n]{X})_\log\to X_\log\to[\overline{\bC}(P)/\bC^+(P)]  \xrightarrow{\phi_{\frac{1}{n}}(P)} [\overline{\bC}(\tfrac{1}{n}P)/\bC^+(\tfrac{1}{n}P)]\to  [\bC(\tfrac{1}{n}P)/\bC^\times(\tfrac{1}{n}P)]
\end{equation}
and 
$$
(\radice[n]{X})_\log\to X_\log\to X_\top.
$$
The claim now follows from the fact that the two maps $(\radice[n]{X})_\log\to [\overline{\bC}(\tfrac{1}{n}P)/\bC^+(\tfrac{1}{n}P)]$ in (\ref{eq:100}) and (\ref{eq:200}) coincide, since the morphism $\phi_{\frac{1}{n}}(P)\colon   [\overline{\bC}(P)/\bC^+(P)]\to  [\overline{\bC}(\tfrac{1}{n}P)/\bC^+(\tfrac{1}{n}P)]$ is the inverse of $\phi_n(P)\colon [\overline{\bC}(\tfrac{1}{n}P)/\bC^+(\tfrac{1}{n}P)] \to  [\overline{\bC}(P)/\bC^+(P)]$.
\end{proof}

\begin{remark}
To conclude, let us point out that if we equip every Kato-Nakayama space (or stack) $X_\log$ with its sheaf of rings $\cO_X^\log$ defined in \cite{KN}, then the isomorphism $X_\log\cong (\radice[n]{X})_\log$ is \textbf{not} an isomorphism of ringed topological stacks. This is further explored in \cite{parabolic} (see Section 5.1 in particular), in relation to question $(2)$ mentioned in the introduction.
\end{remark}

\bibliographystyle{alpha}
\bibliography{biblio}

\end{document}